\documentclass{amsart}

\title[$r$-Stable Hypersurfaces in Conformally Stationary Spacetimes]{$r$-Stable Spacelike Hypersurfaces in\\ Conformally Stationary Spacetimes}

\usepackage{graphicx,pstricks,pst-plot}

\newtheorem{theorem}{Theorem}[section]
\newtheorem{lemma}[theorem]{Lemma}
\newtheorem{proposition}[theorem]{Proposition}
\newtheorem{corollary}[theorem]{Corollary}

\theoremstyle{definition}
\newtheorem{definition}[theorem]{Definition}

\theoremstyle{remark}
\newtheorem{remark}[theorem]{Remark}

\numberwithin{equation}{section}

\author{F. Camargo}
\address{Departamento de Matem\'atica e Estat\'{\i}stica,
Universidade Federal de Campina Grande, Campina Grande,
Para\'{\i}ba, Brazil. 58109-970} \email{fernandaecc@dme.ufcg.edu.br}

\author{A. Caminha}
\address{Departamento de Matem\'atica, Universidade Federal do Cear\'a, Fortaleza,
Cear\'a, Brazil. 60455-760} \email{antonio.caminha@gmail.com}

\author{H. de Lima}
\address{Departamento de Matem\'atica e Estat\'{\i}stica,
Universidade Federal de Campina Grande, Campina Grande,
Para\'{\i}ba, Brazil. 58109-970} \email{henrique@dme.ufcg.edu.br}

\author{M. Vel\'asquez}
\address{Departamento de Matem\'atica, Universidade Federal do Cear\'a, Fortaleza,
Cear\'a, Brazil. 60455-760} \email{marcolazarovelasquez@gmail.com}



\thanks{The second author is partially supported by CNPq, Brazil. The third author is partially supported by CNPq/FAPESQ/PPP, Brazil.
The last author is supported by CAPES, Brazil.}

\begin{document}

\maketitle

\begin{abstract}
In this paper we study the r-stability of closed spacelike hypersurfaces with constant $r$-th mean curvature in conformally stationary spacetimes of constant sectional curvature. In this setting, we obtain a characterization of $r-$stability through the analysis of the first eigenvalue of an operator naturally attached to the $r$-th mean curvature. As an application, we treat the case in which the spacetime is the de Sitter space.
\end{abstract}

\section{Introduction}
The notion of stability concerning hypersurfaces of constant mean curvature of Riemannian ambient spaces was first studied by Barbosa and do Carmo in~\cite{BdC:84}, and Barbosa, do Carmo and Eschenburg in~\cite{BdCE:88}, where they proved that spheres are the only stable critical points of the area functional for volume-preserving variations.

In the Lorentz context, in 1993 Barbosa and Oliker~\cite{Barbosa:93} obtained an analogous result, proving that
constant mean curvature spacelike hypersurfaces in Lorentz manifolds are also critical points of the area functional for variations that keep the volume constant. They also computed the second variation formula and showed, for the de Sitter space $\mathbb S_1^{n+1}$, that spheres maximize the area functional for volume-preserving variations.

More recently, Liu and Junlei \cite{Ximin} have characterized the $r$-stable closed spacelike hypersurfaces with constant scalar curvature in the de Sitter space.

The natural generalization of mean and scalar curvatures for an $n-$dimensional hypersurface is the $r$-th mean curvatures $H_r$, for $r=1,\cdots,n$. In fact, $H_1$ is just the mean curvature and $H_2$ defines a geometric quantity which is related to the scalar curvature.

In \cite{Camargo:09}, some of the authors have studied the problem of strong stability (that is, stability with respect to not necessarily volume-preserving variations) for spacelike hypersurfaces with constant $r$-th mean curvature in a Generalized Robertson-Walker (GRW) spacetime, giving a characterization of $r$-maximal and
spacelike slices.

Here, motivated by these works, we consider closed spacelike hypersurfaces with constant $r$-th mean curvature in a wide class of Lorentz manifolds, the so-called {\it conformally stationary spacetimes}, in order to obtain a relation between $r$-stability and the spectrum of a certain elliptic operator naturally attached to the $r$-th mean curvature of the hypersurfaces. Our approach is based on the use of the Newton transformations $P_r$ and their associated second order differential operators $L_r$ (cf. Section \ref{sec:Preliminaries}). More precisely, we prove the following result.

\begin{theorem}
Let $\overline M^{n+1}_c$ be a conformally stationary Lorentz manifold with constant curvature $c$. Suppose that $\overline M^{n+1}_c$ has a closed conformal vector field $V$ and a Killing vector field $W$. Let $x:M^n\rightarrow\overline M^{n+1}_c$ be a closed spacelike hypersurface, with constant, positive $(r+1)$-th
mean curvature $H_{r+1}$ such that
$$\lambda=c(n-r){n\choose r}H_r-nH_1{n\choose r+1}H_{r+1}-(r+2){n\choose r+2}H_{r+2}$$ 
is constant. Assume also that $Div_{\, \overline M}\, V$ does not
vanish on $M^n$. Then $x$ is $r$-stable if and only if $\lambda$ is
the first eigenvalue of $L_r$ on $M^n$.
\end{theorem}

As an application of the previous result, we obtain the following corollary in the de Sitter space.

\begin{corollary}
Let $x:M^n\rightarrow\mathbb S_1^{n+1}$ be a closed spacelike hypersurface, contained in the chronological future (or past) of an equator of $\mathbb S_1^{n+1}$, with positive constant $(r+1)-$th mean curvature such that
$$\lambda=(n-r){n\choose r}H_r-nH{n\choose r+1}H_{r+1}-(r+2){n\choose r+2}H_{r+2}$$ is constant. Then $x$ is $r$-stable if and only if $\lambda$ is the first eigenvalue of $L_r$ on $M^n$.
\end{corollary}

\section{Preliminaries}\label{sec:Preliminaries}

Let $\overline M^{n+1}$ denote a time-oriented Lorentz manifold with Lorentz metric $\overline g=\langle\,\,,\,\,\rangle$, volume element $d\overline M$ and semi-Riemannian connection $\overline\nabla$. In this context, we consider spacelike hypersurfaces $x:M^n\rightarrow\overline M^{n+1}$, namely, isometric immersions from a connected, $n-$dimensional orientable Riemannian manifold $M^n$ into $\overline M$. We let $\nabla$ denote the
Levi-Civita connection of $M^n$.

If $\overline M$ is time-orientable and $x:M^n\rightarrow\overline M^{n+1}$ is a spacelike  hypersurface, then $M^n$ is orientable (cf.~\cite{ONeill:83}) and one can choose a globally defined unit normal vector field $N$ on $M^n$ having the same time-orientation of $\overline M$. Such an $N$ is named {\em future-pointing Gauss map}
of $M^n$. In this setting, let $A$ denote the shape operator of $M$ with respect to $N$, so that at each $p\in M^n$, $A$ restricts to a self-adjoint linear map $A_p:T_pM\rightarrow T_pM$.

For $1\leq r\leq n$, let $S_r(p)$ denote the $r$-th elementary symmetric function on the eigenvalues of $A_p$; this way one gets $n$ smooth functions $S_r:M^n\rightarrow\mathbb R$, such that
$$\det(tI-A)=\sum_{k=0}^n(-1)^kS_kt^{n-k},$$
where $S_0=1$ by definition. If $p\in M^n$ and $\{e_k\}$ is a basis of $T_pM$ formed by eigenvectors of $A_p$, with corresponding eigenvalues $\{\lambda_k\}$, one immediately sees that
$$S_r=\sigma_r(\lambda_1,\ldots,\lambda_n),$$
where $\sigma_r\in\mathbb R[X_1,\ldots,X_n]$ is the $r$-th elementary symmetric polynomial on the indeterminates
$X_1,\ldots,X_n$.

For $1\leq r\leq n$, one defines the $r$-th mean curvature $H_r$ of $x$ by
$${n\choose r}H_r=(-1)^rS_r=\sigma_r(-\lambda_1,\ldots,-\lambda_n).$$
Also, for $0\leq r\leq n$, the $r$-th Newton transformation $P_r$ on $M^n$ is defined by setting $P_0=I$ (the
identity operator) and, for $1\leq r\leq n$, via the recurrence relation
\begin{equation}\label{eq:Newton operators}
P_r=(-1)^rS_rI+AP_{r-1}.
\end{equation}
A trivial induction shows that
$$P_r=(-1)^r(S_rI-S_{r-1}A+S_{r-2}A^2-\cdots+(-1)^rA^r),$$
so that Cayley-Hamilton theorem gives $P_n=0$. Moreover, since $P_r$ is a polynomial in $A$ for every $r$, it is also self-adjoint and commutes with $A$. Therefore, all bases of $T_pM$ diagonalizing $A$ at $p\in M^n$ also diagonalize all of the $P_r$ at $p$. Let $\{e_k\}$ be such a basis. Denoting by $A_i$ the restriction of $A$ to $\langle e_i\rangle^{\bot}\subset T_p\Sigma$, it is easy to see that
$$\det(tI-A_i)=\sum_{k=0}^{n-1}(-1)^kS_k(A_i)t^{n-1-k},$$
where
$$S_k(A_i)=\sum_{\stackrel{1\leq j_1<\ldots<j_k\leq n}{j_1,\ldots,j_k\neq i}}\lambda_{j_1}\cdots\lambda_{j_k}.$$

With the above notations, it is also immediate to check that $P_re_i=(-1)^rS_r(A_i)e_i$, and hence (cf. Lemma 2.1
of~\cite{Barbosa:97})
\begin{equation}\label{eq:tracos dos Pr}
 \begin{split}
&{\rm tr}(P_r)=(-1)^r(n-r)S_r=b_rH_r;\\
&{\rm tr}(AP_r)=(-1)^r(r+1)S_{r+1}=-b_rH_{r+1};\\
&{\rm tr}(A^2P_r)=(-1)^r(S_1S_{r+1}-(r+2)S_{r+2}),
 \end{split}
\end{equation}
where $b_r=(n-r){n\choose r}$.

Associated to each Newton transformation $P_r$ one has the second order linear differential operator $L_r:\mathcal
D(M)\rightarrow\mathcal D(M)$, given by
$$L_r(f)={\rm tr}(P_r\,\text{Hess}\,f).$$
For instance, when $r=0$, $L_r$ is simply the Laplacian operator.

According to \cite{Alias:07}, if $\overline M^{n+1}$ is of constant sectional curvature, then $P_r$ is divergence-free and, consequently,
$$L_r(f)={\rm div}(P_r\nabla f).$$

If $x$ is as above, a {\em variation} of it is a smooth mapping 
$$X:M^n\times(-\epsilon,\epsilon)\rightarrow\overline M^{n+1}$$
satisfying the following conditions:
\begin{enumerate}
\item[(1)] For $t\in(-\epsilon,\epsilon)$, the map $X_t:M^n\rightarrow\overline M^{n+1}$ given by $X_t(p)=X(t,p)$ is a spacelike immersion such that $X_0=x$.
\item[(2)] $X_t\big|_{\partial M}=x\big|_{\partial M}$, for all $t\in(-\epsilon,\epsilon)$.
\end{enumerate}

In all that follows, we let $dM_t$ denote the volume element of the metric induced on $M$ by $X_t$ and $N_t$ the unit normal vector field along $X_t$.

The {\em variational field} associated to the variation $X$ is the vector field $\frac{\partial X}{\partial t}\Big|_{t=0}$. Letting $f=-\langle\frac{\partial X}{\partial t},N_t\rangle$, we get
\begin{equation}\label{eq:decomposition of variational vector field}\frac{\partial X}{\partial t}=fN_t+\left(\frac{\partial X}{\partial t}\right)^{\top},\end{equation} where $\top$ stands for tangential components. 

The {\em balance of volume} of the variation $X$ is the function $\mathcal V:(-\epsilon,\epsilon)\rightarrow\mathbb R$ given by 
$$\mathcal V(t)=\int_{M\times[0,t]}X^*(d\overline M),$$
and we say $X$ is {\em volume-preserving} if $\mathcal V$ is constant.

From now on, we will consider only closed spacelike hypersurface $x:M^n\rightarrow\overline M^{n+1}$. The following lemma is enough known and can be found in (cf.~\cite{Xin:03}).

\begin{lemma}\label{lemma:first variation}
Let $\overline M^{n+1}$ be a time-oriented Lorentz manifold and $x:M^n\rightarrow\overline M^{n+1}$ a closed spacelike hypersurface. If 
$X:M^n\times(-\epsilon,\epsilon)\rightarrow\overline M^{n+1}$ is a
variation of $x$, then
$$\frac{d\mathcal V}{dt}=\int_MfdM_t.$$
In particular, $X$ is volume-preserving if and only if $\int_MfdM_t=0$ for all $t$.
\end{lemma}

We remark that Lemma 2.2 of~\cite{BdCE:88} remains valid in the Lorentz context, i.e., if $f_0:M\rightarrow\mathbb R$ is a smooth function such that $\int_Mf_0dM=0$, then there exists a volume-preserving variation of $M$ whose variational field is $f_0N$.

In order to extend~\cite{Barbosa:97} to the Lorentz setting, we let the {\em $r$-area functional} $\mathcal
A_r:(-\epsilon,\epsilon)\rightarrow\mathbb R$ associated to the variation $X$ be given by
$$\mathcal A_r(t)=\int_MF_r(S_1,S_2,\ldots,S_r)dM_t,$$
where $S_r=S_r(t)$ and $F_r$ is recursively defined by setting $F_0=1$, $F_1=-S_1$ and, for $2\leq r\leq n-1$,
$$F_r=(-1)^rS_r-\frac{c(n-r+1)}{r-1}F_{r-2}.$$
We notice that if $r=0$, the functional $\mathcal A_0$ is the classical area functional.

The next step is the Lorentz analogue of Proposition 4.1 of~\cite{Barbosa:97}. From Lemma 2.2 in \cite{Camargo:09} we obtain the following result.

\begin{lemma}\label{lemma:computing the time-derivative of Sr}
Let $x:M^n\rightarrow\overline M^{n+1}_c$ be a closed spacelike hypersurface of the time-oriented Lorentz manifold $\overline M^{n+1}_c$ with constant curvature $c$, and let $X:M^n\times(-\epsilon,\epsilon)\rightarrow\overline M^{n+1}_c$ be a variation of $x$. Then,
\begin{equation}\label{eq:differentiation of $S_r$}
\frac{\partial S_{r+1}}{\partial t}=(-1)^{r+1}\left[L_{r}f+c{\rm tr}(P_{r})f-{\rm tr}(A^2P_{r})f\right]+\langle\left(\frac{\partial X}{\partial t}\right)^{\top},\nabla S_{r+1}\rangle.
\end{equation}
\end{lemma}

The previous lemma allows us to compute the first variation of the $r$-area functional.

\begin{proposition}\label{prop:first variation}
Under the hypotheses of Lemma~\ref{lemma:computing the time-derivative of Sr}, if $X$ is a variation of $x$, then
\begin{equation}\label{eq:first variation}
\mathcal A_r'(t)=\int_M[(-1)^{r+1}(r+1)S_{r+1}+c_r]f\,dM_t,
\end{equation}
where $c_r=0$ if $r$ is even and $c_r=-\frac{n(n-2)(n-4)\ldots(n-r+1)}{(r-1)(r-3)\ldots 2}(-c)^{(r+1)/2}$ if $r$ is odd.
\end{proposition}

The proof of this result is a straightforward consequence of Proposition 2.3 in \cite{Camargo:09}.

In order to characterize spacelike immersions of constant $(r+1)-$th mean curvature, let $\lambda$ be a real constant and $\mathcal J_r:(-\epsilon,\epsilon)\rightarrow\mathbb R$ be the {\em Jacobi functional} associated to the variation $X$, i.e., 
$$\mathcal J_r(t)=\mathcal A_r(t)-\lambda\mathcal V(t).$$
As an immediate consequence of (\ref{eq:first variation}) we get
$$\mathcal J_r'(t)=\int_M[b_rH_{r+1}+c_r-\lambda]fdM_t,$$
where $b_r=(r+1){n\choose r+1}$. Therefore, if we choose $\lambda=c_r+b_r\overline H_{r+1}(0)$, where
$$\overline H_{r+1}(0)=\frac{1}{\mathcal A_0(0)}\int_MH_{r+1}(0)dM$$
is the mean of the $(r+1)$-th curvature $H_{r+1}(0)$ of $M$, we arrive at
$$\mathcal J_r'(t)=b_r\int_M[H_{r+1}-\overline H_{r+1}(0)]fdM_t.$$
Hence, a standard argument (cf.~\cite{BdC:84}) shows that $M$ is a critical point of $\mathcal J_r$ for all variations of $x$ if and only if $M$ has constant $(r+1)$-th mean curvature.

We wish to study spacelike immersions $x:M^n\rightarrow\overline M^{n+1}$ that maximize $\mathcal A_r$ for all volume-preserving variations $X$ of $x$. The above dicussion shows that $M$ must have constant $(r+1)$-th mean curvature and, for such an $M$, one is naturally lead to compute the second variation of $\mathcal A_r$. This
motivates the following

\begin{definition}
Let $\overline M^{n+1}_c$ be a time-oriented Lorentz manifold of constant curvature $c$, and $x:M^n\rightarrow\overline M^{n+1}$ be a closed spacelike hypersurface having constant $(r+1)$-th mean
curvature. We say that $x$ is $r$-stable if $\mathcal A_r''(0)\leq 0$, for all volume-preserving variation of $x$.
\end{definition}

\begin{remark}
Let $x:M^n\rightarrow\overline M^{n+1}_c$ be a closed spacelike hypersurface with constant $(r+1)$-th mean curvature and denote by $\mathcal G$ the set of differential functions $f:M^n\rightarrow\mathbb R$ with $\int_MfdM_t=0$. Just as
\cite{Ximin} we can establish the following criterion for stability: {\it $x$ is $r$-stable if and only if $\mathcal J_r''(0)\leq 0$, for all $f\in \mathcal G$}.
\end{remark}

The sought formula for the second variation of $\mathcal J_r$ is another straightforward consequence of Proposition~\ref{prop:first variation}.

\begin{proposition} Let $x:M^n\rightarrow\overline M^{n+1}_c$ be a closed spacelike hypersurface of the time-oriented Lorentz manifold $\overline M^{n+1}_c$, having constant $(r+1)$-mean curvature $H_{r+1}$. If $X:M^n\times(-\epsilon,\epsilon)\rightarrow\overline M^{n+1}_c$ is a variation of $x$, then $J_r''(0)$ is given by
\begin{equation}\label{eq:second formula of variation}
\mathcal J_r''(0)(f)=(r+1)\int_M\left[L_r(f)+\{c{\rm tr}(P_r)-{\rm
tr}(A^2P_r)\}f\right]fdM.
\end{equation}
\end{proposition}

\section{A Characterization of $r$-Stable Spacelike Hypersurfaces}

As in the previous section, let $\overline M^{n+1}$ be a Lorentz manifold. A vector field $V$ on $\overline M^{n+1}$ is said to be {\em conformal} if
\begin{equation}
\mathcal L_V\langle\,\,,\,\,\rangle=2\psi\langle\,\,,\,\,\rangle
\end{equation}
for some function $\psi\in C^{\infty}(\overline M)$, where $\mathcal L$ stands for the Lie derivative of the Lorentz metric of $\overline M$. The function $\psi$ is called the {\em conformal factor} of $V$.

Since $\mathcal L_V(X)=[V,X]$ for all $X\in\mathcal X(\overline M)$, it follows from the tensorial character of $\mathcal L_V$ that $V\in\mathcal X(\overline M)$ is conformal if and only if
\begin{equation}\label{eq:1.1}
\langle\overline\nabla_XV,Y \rangle+\langle X,\overline\nabla_YV\rangle=2\psi\langle X,Y\rangle,
\end{equation}
for all $X,Y\in\mathcal X(\overline M)$. In particular, $V$ is a Killing vector field relatively to $\overline g$ if and only if $\psi\equiv 0$. Observe that the function $\psi$ can be characterized as 
$$\psi=\frac{1}{n+1}Div_{\, \overline M}\, V.$$

An interesting particular case of a conformal vector field $V$ is that in which $\overline\nabla_XV=\psi X$ for all $X\in\mathcal X(\overline M)$; in this case we say that $V$ is closed, an allusion to the fact that its dual $1$-form is closed.

Any Lorentz manifold $\overline M^{n+1}$, possessing a globally defined, timelike conformal vector field is said to be a {\em conformally stationary spacetime}.

In what follows we need a formula first derived in~\cite{Alias:07}. As stated below, it is the Lorentz version of the one stated and proved in~\cite{Barros:09}.

\begin{lemma}\label{lemma:Lr of conformal vector field}
Let $\overline M^{n+1}_c$ be a conformally stationary Lorentz manifold having constant curvature $c$ and conformal vector field $V$. Let also $x:M^n\rightarrow\overline M^{n+1}_c$ be a spacelike hypersurface of $\overline M^{n+1}_c$ and $N$ be a future-pointing Gauss map on $M^n$. If $\eta=\langle V,N\rangle$, then
\begin{eqnarray}\label{eq:Laplacian formula_I}
L_r(\eta)&=&\{{\rm tr}(A^2P_r)-c\,{\rm tr}(P_r)\}\eta -b_rH_rN(\psi)\\
&&+b_rH_{r+1}\psi+\frac{b_r}{r+1}\langle V,\nabla
H_{r+1}\rangle,\nonumber
\end{eqnarray}
where $\psi:\overline M^{n+1}\rightarrow\mathbb R$ is the conformal factor of $V$, $H_j$ is the $j$-th mean curvature of $M^n$ and $\nabla H_j$ stands for the gradient of $H_j$ on $M^n$.
\end{lemma}

In particular, we obtain the following

\begin{corollary}\label{corollary:Lr of killing fields}
Let $\overline M^{n+1}_c$ be a conformally stationary Lorentz manifold having constant curvature $c$ and Killing vector field $W$. Let also $x:M^n\rightarrow\overline M^{n+1}_c$ be a spacelike hypersurface having constant $(r+1)$-th mean curvature $H_{r+1}$, $N$ be a future-pointing Gauss map on $M^n$ and $\eta=\langle W,N\rangle$, then
$$L_r(\eta)+\{c\,{\rm tr}(P_r)-{\rm tr}(A^2P_r)\}\eta =0.$$
In particular, if $x:M^n\rightarrow\overline M^{n+1}_c$ is a closed spacelike hypersurface with constant $(r+1)$-th mean curvature such that $\lambda=c{\rm tr}(P_r)-{\rm tr}(A^2P_r)$ is constant, then $\lambda$ is an eigenvalue of the operator $L_r$ in $M^n$ with eigenfunction $\eta$.
\end{corollary}

\begin{remark}\label{remark:ellipticity} Assuming that the conformal vector field $V$ is closed and such that
$Div_{\, \overline M}\, V$ does not vanish on $M^n$, then there exists an elliptic point in $ M^n$ (cf.~ Corollary 5.5 of~\cite{ABC:03}). Moreover, if $M^n$ has an elliptic point and $H_{r+1}> 0$ on $M$, for $2\leq r\leq n-1$, then $L_r$ is elliptic (cf.~Lemma 3.3 of \cite{Alias:07}). In the case $r=1$, the hypothesis $H_2>0$ garantees the ellipticity of $L_1$ without the additional assumption on the existence of an elliptic point (cf.~Lemma 3.2 of \cite{Alias:07}).
\end{remark}

We can now state and prove our main result.

\begin{theorem}\label{thm:r estability eigenvalue}
Let $\overline M^{n+1}_c$ be a conformally stationary Lorentz manifold with constant curvature $c$. Suppose that $\overline M^{n+1}_c$ has a closed conformal vector field $V$ and a Killing vector field $W$. Let $x:M^n\rightarrow\overline M^{n+1}_c$ be a closed spacelike hypersurface, with positive constant $(r+1)$-th
mean curvature $H_{r+1}$ such that
$$\lambda=c(n-r){n\choose r}H_r-nH_1{n\choose r+1}H_{r+1}-(r+2){n\choose r+2}H_{r+2}$$ 
is constant. Assume also that $Div_{\, \overline M}\, V$ does not vanish on $M^n$. Then $x$ is $r$-stable if and only  if $\lambda$ is the first eigenvalue of $L_r$ on $M^n$.
\end{theorem}

\begin{proof}
From Remark \ref{remark:ellipticity} the operator $L_r$ is elliptic. On the other hand, by using the formulas (\ref{eq:tracos dos Pr}), it is easy to show that $\lambda=c{\rm tr}(P_r)-{\rm tr}(A^2P_r)$. Therefore, since that $\lambda$ is constant and $W$ is a Killing field on $\overline M^{n+1}_c$, Corollary \ref{corollary:Lr of killing fields} guarantees that $\lambda$ is in the spectrum of $L_r$.

Let $\lambda_1$ be the first eigenvalue of $L_r$ on $M^n$. If
$\lambda=\lambda_1$, then the variational characterization of $\lambda_1$ gives
$$\lambda=\min_{f\in \, \mathcal G\setminus \{0\}} \, \frac{-\int_M fL_r(f)dM}{\int_M f^2dM}.$$
It follows that, for any $f\in \mathcal G$,
\begin{eqnarray*}
 \mathcal J_r''(0)(f)&=&(r+1)\int_M\{fL_r(f)+\lambda f^2\}dM\\
&\leq&(r+1)(-\lambda+\lambda)\int_Mf^2dM=0,
\end{eqnarray*}
and $x$ is $r$-stable.

Now suppose that $x$ is $r-$stable, so that $\mathcal J_r''(0)(f)\leq 0$ for all $f\in\mathcal G$. Let $f$ be an eigenfunction associated to the first eigenvalue $\lambda_1$ of $L_r$. As was already observed, there exists a volume-preserving variation of $M$ whose variational field is $fN$. Consequently, by (\ref{eq:second formula of variation}) we get
$$0\geq\mathcal J_r''(0)(f)=(r+1)(-\lambda_1+\lambda)\int_Mf^2dM$$
and therefore $\lambda_1=\lambda$, since that $\lambda_1\leq \lambda$.
\end{proof}

\section{Applications to GRW spacetimes}

A particular class of conformally stationary spacetimes is that of {\em generalized Robertson-Walker} spacetimes, or {\em GRW} for short (cf.~\cite{ABC:03}), namely, warped products $\overline M^{n+1}=-I\times_{\phi}F^n$, where $I\subseteq\mathbb R$ is an interval with the metric $-ds^2$, $F^n$ is an $n$-dimensional Riemannian manifold and $\phi:I\rightarrow\mathbb R$ is positive and smooth. For such a space, let $\pi_I:\overline M^{n+1}\rightarrow I$
denote the canonical projection onto $I$. Then the vector field 
$$V=(\phi\circ\pi_I)\frac{\partial}{\partial s}$$
is a conformal, timelike and closed, with conformal factor $\psi=\phi'$, where the prime denotes differentiation with respect to $s$. Moreover (cf.~\cite{Montiel:99}), for $s_0\in I$, the (spacelike) leaf $M_{s_0}^n=\{s_0\}\times F^n$ is totally umbilical, with umbilicity factor $-\frac{\phi'(s_0)}{\phi(s_0)}$ with respect to the future-pointing unit normal vector field $N$.

If $\overline M^{n+1}=-I\times_{\phi}F^n$ is a GRW and $x:M^n\rightarrow\overline M^{n+1}$ is a complete spacelike
hypersurface of $\overline M^{n+1}$, such that $\phi\circ\pi_I$ is limited on $M^n$, then $\pi_F\big|_M:M^n\rightarrow F^n$ is necessarily a covering map (cf.~\cite{ABC:03}). In particular, if $M^n$ is closed then $F^n$ is automatically closed.

Also, recall (cf.~\cite{ONeill:83}) that a GRW as above has constant sectional curvature $c$ if and only if $F$ has constant sectional curvature $k$ and the warping function $\phi$ satisfies the ODE
$$\frac{\phi''}{\phi}=c=\frac{(\phi')^2+k}{\phi^2}.$$

In this setting, from Theorem~\ref{thm:r estability eigenvalue} we obtain the following

\begin{corollary}\label{corollary:GRW spacetime} Let $x:M^n\rightarrow -I\times_{\phi}F^n$ be a closed spacelike
hypersurface with constant $(r+1)$-th mean curvature $H_{r+1}>0$. Suppose also that $-I\times_{\phi}F^n$ is of constant curvature $c$, has a Killing vector field and $\phi'$ does not vanish on $M^n$. If 
$$\lambda=c(n-r){n\choose r}H_r-nH{n\choose r+1}H_{r+1}-(r+2){n\choose r+2}H_{r+2}$$ 
is constant, then $x$ is $r$-stable if and only if $\lambda$ is the first eigenvalue of $L_r$ on $M^n$.
\end{corollary}

A particular example of GRW spacetime is de Sitter space. More precisely, let $\mathbb L^{n+2}$ denote the $(n+2)$-dimensional Lorentz-Minkowski space ($n\geq 2$), that is, the real vector space $\mathbb R^{n+2}$, endowed with the Lorentz metric
$$\left\langle v,w\right\rangle ={\displaystyle\sum\limits_{i=1}^{n+1}}v_{i}w_{i}-v_{n+2}w_{n+2},$$
for all $v,w\in\mathbb{R}^{n+2}$. We define the $\left(n+1\right)$-dimensional de Sitter space $\mathbb S_1^{n+1}$
as the following hyperquadric of $\mathbb{L}^{n+2}$
$$\mathbb S_1^{n+1}=\left\{p\in\mathbb L^{n+2}:\left\langle p,p\right\rangle=1\right\}.$$
From the above definition it is easy to show that the metric induced from $\left\langle\,\,,\,\right\rangle $ turns $\mathbb S_1^{n+1}$ into a Lorentz manifold with constant sectional curvature $1$.

Choose a unit timelike vector $a\in\mathbb L^{n+2}$, then $V(p)=a-\langle p, a\rangle p$, $p\in\mathbb S_1^{n+1}$ is a
conformal and closed timelike vector field. It foliates the de Sitter space by means of umbilical round spheres $M_{\tau}=\{p\in \mathbb S_1^{n+1}:\langle p,a\rangle=\tau\}$, $\tau\in\mathbb R$. The level set given by $\{p\in\mathbb S_1^{n+1}:\langle p,a\rangle=0\}$ defines a round sphere of radius one which is a totally geodesic hypersurface in $\mathbb S_1^{n+1}$. We will refer to that sphere as the equator of $\mathbb S_1^{n+1}$ determined by
$a$. This equator divides the de Sitter space into two connected components, the chronological future which is given by
$$\{p\in \mathbb S_1^{n+1}: (a ,p) < 0\},$$
and the chronological past, given by
$$\{p\in \mathbb S_1^{n+1}: (a ,p) > 0\}.$$

In the context of warped products, the de Sitter space can be thought of as the following GRW
$$\mathbb S_1^{n+1}=-\mathbb R\times_{\cosh s}\mathbb{S}^n,$$
where $\mathbb S^n$ means Riemannian unit sphere. We observe that there is a lot of possible choices for the unit timelike vector $a\in\mathbb L^{n+2}$ and, hence, a lot of ways to describe $\mathbb S_1^{n+1}$ as such a GRW (cf.~\cite{Montiel:99}, Section $4$). We notice that in this model, the equator of $\mathbb S_1^{n+1}$ is the
slice $\{0\}\times\mathbb{S}^n$ and, consequently, $\phi'(s)=\sinh s$ vanishes only on this slice. Finally, the vector field
$$V=\phi'(s)\frac{\partial}{\partial s}=(\sinh s)\frac{\partial}{\partial s}$$
is conformal, timelike and closed in $\mathbb S_1^{n+1}$.

In order to rewrite Theorem \ref{thm:r estability eigenvalue} for the case of closed spacelike hypersurfaces immersed
in de Sitter space, we recall some facts. 
\begin{enumerate}
\item[(a)] Killing vector fields in de Sitter space $\mathbb S_1^{n+1}$ can be constructed by fixing two vectors $u$ and $v$ in the Lorentz-Minkowski space $\mathbb L^{n+2}$ and a non-zero constant $k\in \mathbb R$, and considering the vector field $W=k\{\langle u,\cdot \rangle v -\langle v,\cdot \rangle u\}$. Geometrically, $W(x)$ determines an orthogonal direction to the position vector $x$ on the subspace spanned by $u$ and $v$ (cf.~ Example 1 of \cite{de Lima:07}).
\item[(b)] Let $x:M^n\rightarrow\mathbb S_1^{n+1}$ be a closed spacelike hypersurface with positive constant $(r+1)-$th mean curvature. Assuming that $M^n$ is contained in the chronological future (or past) of the equator of $S_1^{n+1}$ then $Div_{\, \overline M}\, V$ does not vanish on $M^n$. Also, there exists an elliptic point in
$M^n$ (cf.~Theorem 7 of \cite{Aledo}) and, if $H_{r+1} > 0$ on $M$ for some $2\leq r\leq n-1$, then, for all $1\leq j\leq r$, the operator $L_j$ is elliptic (cf.~Lemma 3.3 of \cite{Alias:07}). In the case of $L_1$, it is sufficient to require that $R <c$ (cf.~Lemma 3.2 of ~\cite{Alias:07}).
\end{enumerate}

We can now state the following corollary of Theorem~\ref{thm:r estability eigenvalue}.

\begin{corollary}
Let $x:M^n\rightarrow\mathbb S_1^{n+1}$ be a closed spacelike hypersurface, contained in the chronological future (or past) of an equator of $\mathbb S_1^{n+1}$, with positive constant $(r+1)$-th mean curvature such that
$$\lambda=(n-r){n\choose r}H_r-nH_1{n\choose r+1}H_{r+1}-(r+2){n\choose r+2}H_{r+2}$$ 
is constant. Then $x$ is $r$-stable if and only if $\lambda$ is the first eigenvalue of $L_r$ on $M^n$.
\end{corollary}

\begin{remark}
We remark that the round spheres of $\mathbb S_1^{n+1}$ are $r$-stable (cf.~\cite{Brasil:03}, Proposition 2).
\end{remark}

\section*{Acknowledgements}
This work was started when the fourth author was visiting the Departamento de Matem\'{a}tica e Estat\'{\i}stica of the
Universidade Federal de Campina Grande. He would like to thank this institution for its hospitality.

\end{document}